\documentclass[11pt, a4paper,two side]{article}

\usepackage[a4paper, right=.95 in, left=.95 in]{geometry}
\usepackage{amsmath,amssymb, amsthm, bbm, authblk, mathtools, wasysym}

\numberwithin{equation}{section}

\newtheorem{thm}{Theorem}[section]

\newtheorem{defin}[thm]{Definition}
\newtheorem{lem}[thm]{Lemma}
\newtheorem{prop}[thm]{Proposition}
\newtheorem{remark}[thm]{Remark}

\newenvironment{enum(a)}{\begin{enumerate}}{\end{enumerate}}
\newenvironment{enum(i)}{\begin{enumerate}}{\end{enumerate}}

\newcommand{\CCC}{{\mathbb{C}}}

\newcommand{\EEE}{{\mathbb{E}}}

\newcommand{\NNN}{{\mathbb{N}}}

\newcommand{\RRR}{{\mathbb{R}}}

\newcommand{\ZZZ}{{\mathbb{Z}}}

\newcommand{\cA}{{\mathcal{A}}}
\newcommand{\cB}{{\mathcal{B}}}
\newcommand{\cC}{{\mathcal{C}}}

\newcommand{\cE}{{\mathcal{E}}}
\newcommand{\cF}{{\mathcal{F}}}
\newcommand{\cG}{{\mathcal{G}}}

\newcommand{\cL}{{\mathcal{L}}}

\newcommand{\cO}{{\mathcal{O}}}

\newcommand{\cT}{{\mathcal{T}}}

\newcommand{\cX}{{\mathcal{X}}}

\newcommand{\gra}{{\alpha}}

\newcommand{\grg}{{\gamma}}
\newcommand{\grd}{{\delta}}

\newcommand{\grl}{{\lambda}}

\newcommand{\grr}{{\rho}}
\newcommand{\grs}{{\sigma}}

\newcommand{\gro}{{\omega}}

\newcommand{\grD}{{\Delta}}

\newcommand{\grL}{{\Lambda}}

\newcommand{\1}{{\mathbbm{1}}}
\newcommand{\dd}{{\mathrm{d}}}

\title{Long-range orientational order of a random near lattice hard sphere and hard disk process}
\author{Alexisz Tam\'as Ga\'al}
\affil{CIMS, New York University 251 Mercer St, New York, NY 10012 USA\\
E-mail: gaal@cims.nyu.edu}
\date{first version: April 23, 2018\\
	second version: September 16, 2019}

\usepackage{fancyhdr}
\begin{document}
\maketitle

\small
\begin{abstract}
We show that a point process of hard spheres exhibits long-range orientational order. This process is designed to be a random perturbation of a three-dimensional lattice that satisfies a specific rigidity property, examples include the FCC and HCP lattices. We also define two-dimensional near-lattice processes by local, geometry dependent hard disk conditions. Earlier results about existence of long-range orientational order carry over and we obtain the existence of infinite-volume measures on two-dimensional point configurations that turn out to follow the orientation of a fixed triangular lattice arbitrary closely.

\textbf{Subject classif.:}  60K35; 82B21
\end{abstract}

\normalsize

\section{Introduction}

Random hard disk and hard sphere processes are one of the most easily defined, physically interesting point processes. Rigorous mathematical results about their behavior at high intensity are limited to two-dimensional systems. It remains an open question whether phase transition with possibly orientational symmetry breaking occurs in either two or three dimensions. If breaking of rotational symmetry in either of these models could be shown, it gave rise to speculation whether such simple pair interaction could result in crystallization phenomena. In order to simplify the models, we exclude cavities and other crystal defects from the models and study random hard disk and sphere processes that are locally crystals. In our models, being locally crystal implies being a crystal on a long range. There is a lower bound on orientational correlation that is uniform in the distance and this bound can be made arbitrarily large by taking very "tight" boundary conditions. Theorem \ref{thm1} is the first rigorous result about hard sphere long-range orientational order to our knowledge.

In the previous work \cite{AG14}, we considered hard disk processes with disks of radius $1/2$ that have the structure of a triangular lattice and neighboring disks have an upper bound on their distance. We showed the existence of a natural "uniform" measures on these allowed configurations that exhibit uniform long-range orientational order. In the first half of this work, we show that the same arguments apply to some three-dimensional lattices. In the second half, we show that the result in the two-dimensional case can be formulated independently of an underlying triangular lattice structure that was explicitly present in the definition of the probability measures in \cite{AG14}. Thus we show that being a crystal locally implies being a crystal on the long range in this particular model. We only require the local, geometry dependent condition that every point has exactly six points in an annulus with radii $1$ and $1+\gra$ around them. We will have the parameter $\gra$ in both sections that gives the maximal distance of neighboring points. This $\gra$ needs to be sufficiently small so that some local conditions are fulfilled, however it is on the macroscopic order of about $1/2$, so not particularly small. Fluctuations from the orientation of a fixed lattice however can be made arbitrary small, in particular they can be made many orders smaller than $\gra$.

Similar but not hard-core models were considered in \cite{MR09} without defects and in \cite{HMR} and \cite{A15} with lattice defects. Introducing bounded, separated missing regions as defects into our two-dimensional model is possible using similar techniques as in \cite{HMR}. For three-dimensions, we think it is possible but we haven't carried it out. Also the techniques of section 3 can possibly be carried out in three-dimension, but an analogue of Lemma \ref{enum} is required together with considering boundary conditions, since in three-dimensions several close-packed lattices are possible analogues of the triangular lattice.

These simplified models with well defined lattice structure and possible defects are motivated by more natural hard sphere models defined with respect to a Poisson point process at a given intensity $z>0$. The set of Gibbs measures for these natural models is defined similarly to our definition of $\cG^z$ in section 3. They are basically sequential limits of Poisson point processes in bounded domains -- as the domains tend to $\RRR^d$ -- conditioned that no pair of points have distance smaller than one. In these natural models, instead of imposing complex geometry dependent interactions, merely hard-core repulsion is required. As a consequence, even at high intensity, all kind of possible lattice defects emerge as soon as the domain gets large enough. It is believed that in dimensions two and greater there are multiple Gibbs measures in $\cG^z$ for high enough intensity $z$. Their structure is believed to differ in the typical relative orientation of nearby points. It is shown in \cite{R07} that in dimension two any of these measures in $\cG^z$ are translational invariant at any intensity $z>0$, and in \cite{R16} a logarithmic lower bound is given on the mean square translational displacement of particles. These results prevent Gibbs measures from having long-range positional order. One strategy of showing that $\cG^z$ is not a singleton for $d\geq 2$ and $z>0$ high enough, is the search for a measure in $\cG^z$ that is not rotational invariant. Existence of such is called the breaking of rotational symmetry (of the energy function). Showing that such measure is supported on a perturbed lattice structure with long-range orientational order would be an even stronger result which is connected to the widely studied crystallization problem, even though the crystallization problem is mostly studied for different interactions.

We would also like to mention the recent result \cite{HT17} that at low intensity disagreement percolation results imply the uniqueness of the Gibbs state. While at high intensity it is shown in \cite{A14} that hard disks percolate with the percolation radius chosen sufficiently big. Percolation is necessary for crystallization, but to our knowledge breaking of rotational symmetry cannot be concluded from it.

\section{The three-dimensional enumerated model}

In this section we show that the arguments of \cite{AG14} can be applied to some three-dimensional lattices to obtain similar results as in \cite{AG14} about long-range orientational order for random perturbations of such lattices. 

\subsection{Configuration space}

We consider three-dimensional lattices with well defined distance between nearest neighbors (to be normalized to 1) that fulfill two conditions. Firstly, the nearest neighbor edges of the lattice have to define a tessellation of $\RRR^3$ by regular tetrahedra and octahedra. Secondly, the lattice has to be translational invariant in three linearly independent directions. We remark that regular tetrahedra and octahedra can be replaced by any rigid polyhedron (a polyhedron with all faces being triangles) that satisfies an analogue of the rigidity estimates in Lemmas \ref{triangle} and \ref{octahedron}, and their volume has positive partial derivatives with respect to their edge lengths. We note that by Cauchy's theorem, the volume is uniquely defined for rigid polyhedra when the edge lengths are given.

Examples of such lattices are the face-centered cubic lattice and the hexagonal close-packed lattice. For definitions see \cite{LPS16}. Note that being translational invariant doesn't mean that the lattice has to be a Bravais lattice, i.e. of the form $\ZZZ n_1 + \ZZZ n_2+\ZZZ n_3$ for some vectors $n_i\in\RRR^3$. Bravais lattices are translational invariant but a union of Bravais lattices might be still translational invariant, however not a Bravais lattice anymore for which the hexagonal close-packed lattice serves as examples.

Let the set $I\subset\RRR^3$ denote one of the lattices that fulfill both criteria. We assume $0\in I$ and think of $I$ as an index set which is going to be used to parametrize countable point configurations in $\RRR^3$. Let $I$ have translational symmetry by the linearly independent vectors $t_1, t_2, t_3\in\RRR^3$ and define the set $T=\ZZZ t_1+ \ZZZ t_2+ \ZZZ t_3$. Define the quotient space $I_n:=I / nT$.  We will think of $I_n$ as a specific set of representatives in the half-open parallelepiped $U_n$ spanned by $nt_1, nt_2, nt_3$, i.e. $U_n=n\{x t_1+ y t_2+ z t_3\ |\ x,y,z\in[0,1)\}$.

A \emph{parametrized point configuration} in $\RRR^3$ is a map $\omega: I\rightarrow \RRR^2$, $x\mapsto \omega(x)$ that determines the point configuration $\{ \omega(x)\ |\ x\in I \}\subset \RRR^3$.  For the set of all parametrized point configurations we introduce the character $\Omega=\{\omega: I\rightarrow \RRR^2\}$. Note that a single point configuration $\{ \omega(x)\ |\ x\in I \}\subset$ can be parametrized by many different $\omega\in\Omega$.

Let $\gra\in(0,1]$ be an arbitrary but fixed real to be fixed later. An \emph{$n$-periodic parametrized point configuration} with edge length $l\in(1, 1+\gra)$ is a parametrized configuration $\omega$ which satisfies the boundary conditions:

\begin{equation}
\omega(x+nt_i)=\omega(x)+lnt_i \quad \textrm{for all}\ x \in I \textrm{ and } i\in\{1,2,3\}.
\label{periodic}
\end{equation}

The set of $n$-periodic parametrized configurations with edge length $l$ is denoted by $\Omega^{per}_{n, l}\subset \Omega$. From now on we will omit the word parametrized because, in this section, we are going to work solely with \emph{point configurations} which are parametrized by $I$. An $n$-periodic configuration is uniquely determined by its values on $I_n$. Therefore, we identify $n$-periodic configurations $\omega\in\Omega^{per}_{n,l}$ with functions $\omega: I_n\rightarrow \RRR^2$.

The bond set $E\subset I\times I$ contains index-pairs with Euclidean distance one; this is $E=\{(x,y)\in I\times I\ |\ |x-y|=1\}$. We set $E_n=E/ nT $, we can think of $E_n$ as a bond set $E_n\subset I_n\times I_n$. Let $\cT$ denote the set of convex polyhedra, as in the definition of $I$, whose edges are in $E$ and provide a tessellation of $\RRR^3$, which is the Delaunay pre-triangulation, see \cite{LPS16}. Define $\cT_n=\cT/ nT$. Each $\triangle \in\cT$ can be triangulated into tetrahedra (not necessarily uniquely), let us fix such a $T$-periodic triangulation of $\cT$. The set of all (necessarily not all regular) tetrahedra created this way define a tessellation of $\RRR^3$ and is denoted by $\textrm{triang}(\cT)$. We define $\textrm{triang}(\cT_n):=\textrm{triang}(\cT)/ nT$. 
\subsection{Probability space}

By definitions of $\Omega$ and $\Omega^{per}_{n, l}$, we have $\Omega=(\RRR^3)^I$ and can identify $\Omega^{per}_{n, l}=(\RRR^3)^{I_n}$. Both sets are endowed with the corresponding product $\sigma$-algebras $\cF=\bigotimes_{x\in I}\cB(\RRR^3)$ and $\cF_n=\bigotimes_{x\in I_n}\cB(\RRR^3)$ where $\cB(\RRR^3)$ denotes the Borel $\sigma$-algebra on each factor.
The event of admissible $n$-periodic configurations $\Omega_{n, l}\subset \Omega^{per}_{n, l}$ is defined by the properties $(\Omega 1)-(\Omega 4)$:

$(\Omega 1) \quad |\omega(x)-\omega(y)|\in (1, 1+\gra)$ for all $(x,y)\in E$.

For $\omega\in\Omega$ we define the extension $\hat\omega: \RRR^3\to\RRR^3$ such that $\hat\omega(x)=\omega(x)$ if $x\in I$. On the closure of a tetrahedron $\triangle\in\textrm{triang}(\cT)$, the map $\hat\omega$ is defined to be the unique affine linear extension of the mapping defined on the corners of that tetrahedron.

$(\Omega 2)\quad$ The map $\hat\omega: \RRR^3\to\RRR^3$ is injective (and thus bijective).

$(\Omega 3) \quad$ The map $\hat\omega$ is almost everywhere orientation preserving, this is to say that $\det(\nabla\hat\omega(x))>0$ for almost every $x\in\RRR^3$ with the Jacobian $\nabla\hat\omega: \RRR^3\to\RRR^{3\times3}$.

$(\Omega 4) \quad$ The image $\hat\omega(\triangle)$ of a polyhedron $\triangle\in\cT$ is a convex polyhedron.

The conditions $(\Omega 3)$ and $(\Omega 4)$ follow from conditions $(\Omega 1)$ and $(\Omega 2)$ up to the sign of the determinant in $(\Omega 3)$ as it was also remarked in \cite{HMR} on page 4. Since the proof is more analytic than stochastic, we also omit the proof and require them as technical conditions. Define the set of \emph{admissible $n$-periodic configurations, with edge length $l$} as

\[
\Omega_{n, l}=\{\omega\in\Omega_{n, l}^{per}\ |\ \omega\ \textrm{satisfies}\ (\Omega1)-(\Omega4)\}.
\]

The set $\Omega_{n, l}$ is open and non-empty subsets of $(\RRR^3)^{I_n}$. The scaled lattice $\omega_l(x)=lx$ for $x\in I$ and $1<l<1+\gra$ is an element of $\Omega_{n, l}$. Any configuration $\gro\in \Omega_{n, l}$ is determined by a finite number of locations in $\RRR^3$. Each property $(\Omega 1)-(\Omega 4)$ is satisfied after small enough perturbation of these locations, therefore any $\gro\in \Omega_{n, l}$ has a neighborhood that is fully contained in $\Omega_{n, l}$, hence the openness of $\Omega_{n, l}$.

Clearly,  $0<\delta_0\otimes\lambda^{I_n \setminus\{0\}} (\Omega_{n, l})<\infty$ with the Lebesgue measure $\lambda$ on $\RRR^3$ and the Dirac measure $\delta_0$ in $0\in\RRR^3$. The lower bound holds because $\Omega_{n, l}^0$ is non-empty and open in $(\RRR^3)^{I_n\setminus\{0\}}$ (similarly to the case of $\Omega_{n, l}$ above); the upper bound is a consequence of the parameter $\gra$ in $(\Omega 1)$. Let the probability measure $P_{n, l}$ be

\[
P_{n, l}(A)=\frac{\delta_0\otimes\lambda^{I_n \setminus\{0\}} (\Omega_{n, l}\cap A)}{\delta_0\otimes\lambda^{I_n \setminus\{0\}} (\Omega_{n, l})}
\]

for any Borel measurable set $A\in \cF_n$, thus $P_{n, l}$ is the uniform distribution on the set $\Omega_{n, l}$ with respect to the \emph{reference measure} $\delta_0\otimes\lambda^{I_n \setminus\{0\}}$. The first factor in this product refers to the component $\omega(0)$ of $\omega\in\Omega$.

\subsection{Result}

We have the following finite-volume result.

\begin{thm} \label{thm1} For $\gra$ sufficiently small one has

\begin{equation}
\lim_{l\downarrow 1}\sup_{n\in\NNN}\sup_{\triangle\in\textrm{triang}(\cT_n)} E_{P_{n,l}}[\ |\nabla\hat\omega(\triangle)-\textnormal{Id} |^2\ ]=0
\label{fnvolume}
\end{equation}

with the constant value of the Jacobian $\nabla\hat\omega(\triangle)$ on the tetrahedron $\triangle$ from the triangulation of $\cT_n$ and some norm $| \cdot |$ on $\RRR^{3\times 3}$.
\end{thm}

The choice of $\gra$ has to be such that the volume of any tetrahedron and octahedron with side lengths in $[1, 1+\gra]$ is uniquely minimized by the regular tetrahedron and octahedron with side length 1 respectively (see proof of Theorem \ref{thm1}). The central argument is going to be the following rigidity theorem from \cite[Theorem 3.1]{Fries}.

\begin{thm}[Friesecke, James and M\"uller]
Let $U$ be a bounded Lipschitz domain in $\RRR^d, \ d\geq 2$. There exists a constant $C(U)$ with the following property: For each $v\in W^{1,2}(U, \RRR^d)$ there is an associated rotation $R\in \textnormal{SO($d$)}$ such that

\begin{equation*}
||\nabla v-R||_{L^2(U)}\leq C(U)||\textnormal{dist}(\nabla v, \textnormal{SO}(d))||_{L^2(U)}.
\end{equation*}

\label{fries}
\end{thm}

This is a generalization of Liouville's theorem, which states that a map is necessarily a rotation whose Jacobian is a rotation in every point of its domain. We are going to set $v=\hat\omega|_{U_n}$ and $U=U_n$ which is a bounded Lipschitz domain. The function $\hat\omega|_{U_n}$  is linear on each triangle $\triangle\in\cT_n$, thus piecewise affine linear on $U_n$. As a consequence, $\hat\omega|_{U_n}$ belongs to the class $W^{1,2}(U_n, \RRR^{3})$. The following remark, which also appears in \cite{Fries} at the end of Section 3, is essential to achieve uniformity in Theorem \ref{fnvolume} in the parameter $n$.

\begin{remark}

The constant $C(U)$ in \emph{Theorem \ref{fries}} is invariant under scaling: $C(\gamma U)=C(U)$ for all $\gamma>0$. Indeed, setting $v_\gamma(\gamma x)=\gamma v(x)$ for $x\in U$, we have $\nabla v_\gamma(\gamma x)=\nabla v(x)$ and hence $||\nabla v_\gamma-R||_{L^2(\gamma U)}=\gamma^{d/2}||\nabla v-R||_{L^2(U)}$ and $||\textnormal{dist}(\nabla v_\gamma, \textnormal{SO}(d))||_{L^2(\gamma U)}=\gamma^{d/2}||\textnormal{dist}(\nabla v, \textnormal{SO}(d))||_{L^2(U)}$. This implies that for the domains $U_n$ $(n\geq 1)$, the corresponding constant $C(U_n)$ can be chosen independently of $n$.
\label{remark}
\end{remark}

\subsection{Proofs}

We are going to show that the $L^2$-distance of the Jacobian $\nabla\hat\omega$ from the scaled identity matrix on $U_n$ can be controlled by the difference of the areas of $\hat\omega(U_n)$ and $U_n$. Because of the periodic boundary conditions, $\lambda(\hat\omega(U_n))$ does not depend on configurations $\omega$ with $(\Omega 2)$, thus it provides a suitable uniform control on the set $\Omega_{n,l}$. Then we show that the expected square distance of $\nabla\hat\omega$ from the scaled identity matrix can be controlled by the the expected square deviation of the polyhedra's edge lengths from one. The one should be associated with the lattice constant of the unscaled lattice.

The following two lemmas from \cite{LPS16} provide the desired rigidity estimate on tetrahedra and octahedra. They state that the distance from $\textrm{SO}(3)$ of a piecewise affine linear map defined on the polyhedron can be controlled by terms that measure how the map deforms the edge lengths of the polyhedron. We conjecture that any convex, rigid polyhedron satisfies such rigidity estimates via Dehn's theorem and the Inverse function theorem. However in this paper, as its main concern is not rigidity theory, we will only consider tetrahedra and octahedra for which these estimates are already proven. Let $|M|=\sqrt{\textnormal{tr}(M^tM)}$ denote the Frobenius norm of a matrix $M\in \RRR^{3\times 3}$ and $|w|$ the Euclidean norm of $w\in\RRR^3$.
\begin{lem}[\cite{LPS16} Lemma 3.2.]
There is a positive constant $C_1$ such that, for all linear maps  $A: \RRR^3 \rightarrow \RRR^3$ with $\textnormal{det}(A)>0$ and $w_1=(1, 0, 0)$, $w_2=(\frac{1}{2}, \frac{\sqrt3}{2}, 0)$, $w_3=w_2-w_1$, $w_4=(\frac{1}{2}, \frac{\sqrt3}{6}, \frac{\sqrt6}{3})$, $w_5=w_4-w_2$, $w_6=w_4-w_1$ and $l\geq1$, the following inequality holds:
\begin{equation}
\textnormal{dist}^2\left(A\ ,\ \textnormal{SO}(3)\right):=\inf_{R\in \textnormal{SO}(3)} \left|A-R\right|^2\leq C_1 \sum_{i=1}^6 (|Aw_i|-1)^2.
\label{triangleeq}
\end{equation}
\label{triangle}
\end{lem}

A similar theorem holds for octahedra. Let $\cO$ denote an octahedron with vertices $P_i$, $i\in\{1,\dots,6\}$, and edges $P_iP_j$ for $i\not= j$ (mod 3).

\begin{lem}[\cite{LPS16} Lemma 3.4.]
There is a constant $C_2>0$ such that
\begin{equation}
\textnormal{dist}^2\left(\nabla u\ ,\ \textnormal{SO}(3)\right)\leq C_2 \sum_{i\not= j \textrm{ (mod 3)}} (|u(P_iP_j)|-1)^2 \quad \textrm{almost everywhere in } \cO,
\label{triangleeq2}
\end{equation}

for every $u \in \cC^0(\cO; \RRR^3)$ such that $u$ is piecewise affine with respect to the triangulation determined by cutting $\cO$ along the diagonal $P_1P_4$, $\det (\nabla u) > 0$ a.e. in $\cO$, and $u(\cO)$ is convex.
\label{octahedron}
\end{lem}

Now, we prove the mentioned estimate, which provides control over the $L^2$-distance of $\nabla\hat\omega$ from the scaled identity matrix in terms of the edge length deviations.

\begin{lem}
For a polyhedron $\triangle\in\cT$, let $\cE(\triangle)$ denote the set of edges of $\triangle$. There is a constant $c>0$ such that for all $n\geq 1$ and $1<l<1+\gra$, the inequality

\begin{equation}
||\ \nabla\hat\omega-l\ \textnormal{Id}\ ||^2_{L^2(U_n)}\leq c \sum_{\triangle\in\cT_n} \sum_{\{x, y\}\in {\cE(\triangle)}} (|\omega(x)-\omega(y)|-1)^2
\label{sideeq}
\end{equation}

holds for all $\omega\in\Omega_{n, l}$, and hence

\begin{equation}
E_{P_{n,l}}[\ ||\ \nabla\hat\omega-l\ \textnormal{Id}\ ||^2_{L^2(U_n)}\ ]\leq c \sum_{\triangle\in\cT_n} \sum_{\{x, y\}\in {\cE(\triangle)}} E_{P_{n,l}}[\ (|\omega(x)-\omega(y)|-1)^2\ ]
\label{lemmaeq2}
\end{equation}

where the $L^2$-norm is defined with respect to the scalar product on $\RRR^{3\times 3}$ that induces the Frobenius norm, and $|\cdot|$ denotes the Euclidean norm on $\RRR^3$.
\label{side}
\end{lem}

Note that the right side in equation (\ref{sideeq}) is strictly positive because of the boundary conditions (\ref{periodic}) and because $l>1$, whereas the left is zero for $\omega=\omega_l\in\Omega^{per}_{n,l}$. Since the measure $P_{n,l}$ is supported on the set $\Omega_{n, l}$, (\ref{lemmaeq2}) follows from (\ref{sideeq}). Also note that $c$ does not depend on $n$.

\begin{proof}

Let $\omega\in\Omega_{n, l}$ and $\cE(\triangle)$ be the set of edges of a polyhedron $\triangle\in\cT_n$. By Lemma \ref{triangle} and Lemma \ref{octahedron} we conclude that on every polyhedron $\triangle\in\cT_n$, we have 

\begin{equation*}
\textrm{dist}^2\left(\nabla\hat\omega|_\triangle ,\ \textrm{SO}(3)\right) \leq \max \{C_1, C_2\} \sum_{\{x, y\}\in {\cE(\triangle)}} (|\omega(x)-\omega(y)|-1)^2
\end{equation*}

where we used $(\Omega 1)$, $(\Omega 3)$ and $(\Omega 4)$ to apply lemmas \ref{triangle} and \ref{octahedron} and with the constants $C_1, C_2$ from lemmas \ref{triangle} and \ref{octahedron}. Orthogonality of functions which are non-zero only on disjoint polyhedra gives

\[
||\ \textrm{dist}(\nabla\hat\omega, \textrm{SO}(3))\ ||^2_{L^2(U_n)}\leq C \sum_{\triangle\in\cT_n} \sum_{\{x, y\}\in {\cE(\triangle)}}(|\omega(x)-\omega(y)|-1)^2
\]

with constant $C=\max\{ C_1,\ C_2\} \max\{\sqrt 2 /12, \sqrt 2 /3\}$ where the second factor is the maximum the volume of a regular tetrahedron and octahedron. Applying Theorem \ref{fries} about geometric rigidity, we find an $R(\omega)\in\textrm{SO}(3)$ such that

\[
||\ \nabla\hat\omega-R(\omega)\ ||^2_{L^2(U_n)}\leq K \ ||\ \textrm{dist}( \nabla\hat\omega, \textrm{SO}(3))\ ||^2_{L^2(U_n)},
\]

with a constant $K>0$ that does not depend on $n$ by Remark \ref{remark}. Due to the periodic boundary conditions (\ref{periodic}), the function $\hat\omega-l\ \textrm{Id}$ is $n$-periodic in the directions $t_1, t_2, t_3$, this is to say

\begin{equation}
\hat\omega(x+nt_i)-l (x+nt_i)=\hat\omega(x)-lx \quad \textrm{for all } x\in \RRR^3 \textrm{ and } \ i \in \{1,2,3\}.
\label{Nperiodic}
\end{equation}

By the fundamental theorem of calculus, the gradient of a periodic function is orthogonal to any constant function, and therefore

\[
||\ \nabla\hat\omega-l\ \textrm{Id}\ ||^2_{L^2(U_n)}+||\ l\ \textrm{Id}-R(\omega)\ ||^2_{L^2(U_n)}=||\ \nabla\hat\omega-R(\omega)\ ||^2_{L^2(U_n)}
\]

by Pythagoras. Since $P_{n,l}$ is supported on the set $\Omega_{n,l}$, the lemma is established with $c=C K$.
\end{proof}

With Lemma \ref{side} we can now prove Theorem \ref{thm1}.

\begin{proof}[Proof of Theorem \ref{thm1}]

A generalization of Heron's formula for tetrahedra gives the volume $\lambda(\triangle)$ of the tetrahedron $\triangle$ with edge lengths $u, v, w, U, V, W$ (opposite edges denoted with the same letter, lower case and capital)

\begin{equation}
\lambda(\triangle)=\frac{\sqrt{(-a+b+c+d)(a-b+c+d)(a+b-c+d)(a+b+c-d)}}{192\ u v w}
\label{heron}
\end{equation}
with
\begin{align*}
   X &= (w - U + v) (U + v + w) \qquad a = \sqrt{x Y Z}\\ 
 x &= (U - v + w) (v - w + U) \qquad b = \sqrt{y Z X}\\
 Y& = (u - V + w) (V + w + u) \qquad c = \sqrt{z X Y}\\
 y &= (V - w + u) (w - u + V) \qquad d =\sqrt{x y z}\\
y &= (V - w + u) (w - u + V)\\
Z &= (v - W + u) (W + u + v)\\
z &= (W - u + v) (u - v + W).\\
\end{align*}
By first order Taylor approximation of (\ref{heron}) at the regular tetrahedron $\triangle_1$, denoting the edge lengths $a_i$, $i\in\{1,\dots,6\}$ we obtain

\begin{equation*}
\lambda(\triangle)-\lambda({\triangle_1})=\frac{1}{12\sqrt{2}}\sum_{i=1}^6(a_i-1)+o\left(\sum_{i=1}^6 |a_i-1|\right)\quad \textrm{as}\ a_i\to1 \textrm{ for all } i.
\end{equation*}

For the octahedron, we obtain $\frac{1}{6\sqrt{2}}$ for the volume derivative in one edge $b_1$ at $b_1=1$ and the remaining $11$ edges fixed at $b_i=1$. This can be achieved by dividing the octahedron into 4 tetrahedrons that all have a common edge $d$ that is a diagonal of the octahedron adjacent to $x$. Using the formula (\ref{heron}) and some elementary geometry of a regular trapezoid to see that $d=\sqrt{x+1}$, we obtain with the regular octahedron $\octagon_1$ of edge length 1:

\begin{equation*}
\lambda(\octagon)-\lambda({\octagon_1})=\frac{1}{6\sqrt{2}}\sum_{i=1}^{12}(b_i-1)+o\left(\sum_{i=1}^{12} |b_i-1|\right)\quad \textrm{as}\ b_i\to1 \textrm{ for all } i.
\end{equation*}

We only need that the partial derivatives of the volume at $\triangle_1$ and $\octagon_1$ are positive. By continuity, in a small neighborhood of the regular polyhedra, increasing one edge length, increases the volume. Therefore we can choose $\gra>0$ from the definition of allowed configurations so small such that the polyhedra of the tessellation obtain minimal volume as the edge lengths go to 1. We choose $c_1>12\sqrt{2}$ and a corresponding $\gra>0$ so small that the inequalities
\begin{align}
\sum_{i=1}^6(a_i-1) \leq c_1 (\lambda(\triangle)-\lambda({\triangle_1})) \nonumber \\
\sum_{i=1}^{12}(b_i-1) \leq c_1 (\lambda(\octagon)-\lambda({\octagon_1}))
\label{area1}
\end{align}

are satisfied whenever $1<a_i<1+\gra$ and $1<b_i<1+\gra$. Let us fix such $c_1>0$ and $\gra>0$ and assume that $\Omega^{per}_{n,l}$ is defined by means of this $\gra.$ Using (\ref{area1}) we can also estimate the squared edge length deviations:

\begin{align}
\sum_{i=1}^6(a_i-1)^2 \leq c_1\ \gra\ (\lambda(\triangle)-\lambda({\triangle_{1}}))\nonumber \\
\sum_{i=1}^{12}(b_i-1)^2 \leq c_1\ \gra\ (\lambda(\octagon)-\lambda({\octagon_1}))
\label{area}
\end{align}

By equation (\ref{sideeq}) from Lemma \ref{side} and ($\ref{area}$), we get an upper bound on $||\nabla\hat\omega-l\ \textnormal{Id}||_{L^2(U_n)}^2$ in terms of the area differences. By summing up the contributions ($\ref{area}$) of the polyhedra $\triangle\in\cT_n$, we conclude for all $\omega\in\Omega_{n,l}$ that

\begin{equation}
||\ \nabla\hat\omega-l\ \textnormal{Id}\ ||_{L^2(U_n)}^2 \leq c_1 \ \gra\ c \sum_{\triangle\in\cT_n}(\lambda(\hat\omega(\triangle))-\lambda({\triangle})).
\label{maxarea}
\end{equation}

As a consequence of $(\Omega 2)$ and the periodic boundary conditions ($\ref{periodic}$), the right hand side in ($\ref{maxarea}$) does not depend on $\omega\in\Omega_{n,l}$. Hence, with $\omega_l\in\Omega_{n,l}$ we can compute

\begin{equation}
\sum_{\triangle\in\cT_n}(\lambda(\hat\omega(\triangle))-\lambda({\triangle}))=\sum_{\triangle\in\cT_n}(\lambda(\hat\omega_l(\triangle))-\lambda({\triangle}))=|U_n|(l^3-1).
\label{maxarea1}
\end{equation}

The combination of the equations (\ref{maxarea}) and (\ref{maxarea1}) gives

\begin{equation}
||\ \nabla\hat\omega-l\ \textnormal{Id}\ ||_{L^2(U_n)}^2 \leq c_1 \ \gra\ c\ |U_n|\ (l^3-1).
\label{maxarea2}
\end{equation}

The reference measure $\delta_0\otimes\lambda^{I_n\setminus\{0\}}$ and the set of allowed configurations $\Omega_{n,l}$ are invariant under under the translations
\[
\psi_b: \Omega^{\textrm{per}}_{n,l} \to \Omega^{\textrm{per}}_{n,l} \quad (\omega(x))_{x\in I}  \mapsto (\omega(x + b) - \omega(b))_{x\in I}
\]
for $b\in T$. As a consequence the matrix valued random variables $\nabla(\hat\omega(\triangle))$ are identically distributed for $\triangle, \widetilde\triangle\in\textrm{triang}(\cT_n)$ such that $\triangle = \widetilde\triangle$ (mod $T$). Thus for any $\triangle\in \textrm{triang}(\cT_1)$ the random variables $\nabla(\hat\omega(\triangle+t))_{t\in T}$ are identically distributed. Therefore
\begin{align*}
E_{P_{n,l}}&[\ ||\ \nabla\hat\omega-l\ \textnormal{Id}\ ||_{L^2(U_n)}^2 \ ] =\sum_{\triangle\in \textrm{triang}(\cT_1)} |U_n(\triangle)| \ E_{P_{n,l}}[\ |\nabla\hat\omega(\triangle)-l\ \textnormal{Id}|^2 \ ] 
\end{align*}
with the regions $U_n(\triangle)$ of $U_n$ taken up by $T$-translates of $\triangle$. Since the proportions $|U_n(\triangle)|/|U_n|$ are independent of $n$ for any $\triangle\in \textrm{triang}(\cT_1)$, this equation together with (\ref{maxarea2}), implies

\begin{equation*}
\lim_{l\downarrow1}\sup_{n\in\NNN}\sup_{\triangle\in\textrm{triang}(\cT_n)} E_{P_{n,l}}[\ |\nabla\hat\omega(\triangle)-l\ \textnormal{Id}|^2 \ ]=0.
\end{equation*}

By means of the triangle inequality, we see that for all $\triangle\in\textrm{triang}(\cT_n)$ and $\omega\in\Omega_{n,l}$

\begin{equation*}
|\nabla\hat\omega(\triangle)- \textnormal{Id}|^2\leq |\nabla\hat\omega(\triangle)-l\ \textnormal{Id}|^2+c_2^2(l-1)^2+2 c_2\ |l-1|\ |\nabla\hat\omega(\triangle)-l\ \textnormal{Id}|
\end{equation*}

with $c_2=|\mathrm{Id}|>0$. For $\omega\in\Omega_{n,l}$, the term $|\nabla\hat\omega(\triangle)-l\ \textnormal{Id}|$ is uniformly bounded for $l\in (1, \gra)$ and $n\in\NNN$, which proves the theorem.
\end{proof}

\section{Two-dimensional model with local geometry dependent interactions}

In this section, we extend the result of \cite{AG14} about long-range orientational order in that we get rid of the a priori enumeration of two-dimensional hard disk configurations by an underlying triangular lattice and merely impose local geometry dependent conditions by means of a Hamiltonian $H$. The conditions impose that hard disks have exactly six neighbors that are not too far away. We show that long-range orientational order carries over to infinite volume Gibbsian point process defined by $H$.
\subsection{Definitions}
Let us cite some definitions from \cite{DDG12}. We equip the plane $\RRR^2$ with its Borel $\sigma$-algebra $\cB(\RRR^2)$ and by $\lambda$ we denote the Lebesgue measure on $(\RRR^2, \cB(\RRR^2))$. The characters $\Lambda$ and $\Delta$ will always denote measurable regions in $\RRR^2$ and the notation $\Delta\Subset\RRR^2$ means that in addition $\Delta$ is bounded. Consider the set $\cX\subset 2^{\left(\RRR^2\right)}$ of locally finite point configurations in $\RRR^2$. That means $X\in \cX$ is a subset $X\subset \RRR^2$ and for any $\Delta\Subset\RRR^2$, the intersection $X_\Delta:=\textrm{pr}_\grD(X):=X\cap \Delta$ has finite cardinality $|X_\Delta|<\infty$. The counting variables $N_\Delta(X):=|X_\Delta|$ generate a $\sigma$-algebra $\cA:=\sigma(N_\Delta : \Delta\Subset\RRR^2 )$ on $\cX$. The union of $X, Y\in \cX$ will be denoted by $XY$, this will be used when defining the configuration $X_\Lambda Y_{\Lambda^c}$ that agrees with $X$ on $\Lambda$ and with $Y$ on the complement of $\Lambda$. In a sequence of set operation, unions $XY$ are to evaluate first in order to reduce brackets. On the measurable space $(\cX, \cA)$, we consider the Poisson point process $\Pi^z$ with intensity $z>0$. The measure $\Pi^z$ is uniquely characterized by the properties that that for all $\Delta\Subset\RRR^2$ under $\Pi^z$: (i) $N_\Delta$ is Poisson distributed with parameter $z\lambda(\Delta)$, and (ii) conditional on $N_\Delta=n$, the $n$ points in $\Delta$ are independently and uniformly distributed on $\Delta$ for each integer $n\geq 1$. Similarly, configurations $\cX_\Lambda=\{X_\Lambda: X\in \cX\}$ in the set $\Lambda$ carry the trace $\sigma$-algebra $\cA_\Lambda':=\cA|_{\cX_\Lambda}$ and the reference measure $\Pi^z_\Lambda$ which is the law of $X_\Lambda$ if $X$ is distributed according to $\Pi^z$. We will also need the pullback of $\cA_\Lambda'$ to $\cX$ defined by $\cA_\grL:=\textrm{pr}_\grL^{-1} \cA_\Lambda'\subset \cA$. Finally, we define the shift group $\Theta=\{\theta_r: r\in \RRR^2\}$, where $\theta_r: \cX\to\cX$ is the translation  by $-r\in\RRR^2$, consequently $N_\Delta(\theta_r X)=N_{\Delta+r}( X)$ for all $\Delta \Subset \RRR^2$.

We fix $\alpha>0$ small enough, the size of $\alpha$ will be specified later. We change the notation of \cite{AG14} from $\epsilon$ to $\alpha$ at this point to emphasize that $\alpha$ is fixed and not particularly small. Let $\Lambda^{1+\alpha}:=\{x\in\RRR^2: \ |x-y|<1+\alpha \textrm{ for some } y\in\Lambda \}$ be the $(1+\gra)$-enlargement of $\Lambda$. For $X\in\cX$ we define the Hamiltonian $H_{\Lambda,Y}$ in $\Lambda$ with boundary condition $Y\in\cX$ by

\[
H_{\Lambda,Y}(X):=\begin{cases}
0 & \textrm{for all }x \in X_{\Lambda}Y_{\Lambda^{1+\alpha}\setminus \Lambda} \textrm{ and } y\in X_\Lambda Y_{\Lambda^c}: |x-y|>1\\
&\textrm{and for all } x \in X_{\Lambda}Y_{\Lambda^{1+\alpha}\setminus \Lambda}: \ | X_\Lambda Y_{\Lambda^c} \cap A_{1, 1+\alpha}(x) |=6 \\
\infty & \textrm{otherwise}.
\end{cases}
\]

This is to say that $H_{\Lambda,Y}(X)\in\{0,\infty\}$ takes the value 0 if and only if every point of $X_{\Lambda^{1+\alpha}}$ has distance greater than one from points in $X_\Lambda Y_{\Lambda^c}$ and has exactly six $X_\Lambda Y_{\Lambda^c}$-neighbors in the annulus $A_{1, 1+\alpha}(x)=\{y\in\RRR^2: |y-x|\in(1,1+\alpha)\}$, otherwise $H$ is defined to be infinity. Note that the only part of the boundary condition $Y$ relevant for $H_{\Lambda,Y}(X)$ is in the region $\Lambda^{2(1+\alpha)}\setminus \Lambda$.

\begin{defin} We define the partition function $Z^z_{\Lambda, Y}$ by

\[
Z^z_{\Lambda, Y}:=\Pi^z_\Lambda\{ X_\Lambda: H_{\Lambda,Y}(X_\Lambda)=0\}=\int e^{-H_{\Lambda,Y}(X)} \Pi^z_\Lambda(dX).
\]

We call a boundary condition $Y\in\cX$ admissible for the region $\Lambda \Subset \RRR^2$ if $0<Z^z_{\Lambda, Y}$. We write $\cX_*^{\Lambda, z}$ for the set of all these $Y$.
\end{defin}

The set of admissible boundary conditions $\cX_*^{\Lambda, z}$ is never empty as the $l\in(1,1+\alpha)$ multiply of a triangular lattice with lattice constant one is always in $\cX_*^{\Lambda, z}$. We note that $H_{\Lambda,Y}(\emptyset)=0$ for $Y_{\Lambda^{1+\alpha}}=\emptyset$ and also for specifically chosen $\Lambda$ and possibly non-empty $Y$. The partition function $Z^z_{\Lambda, Y}$ is zero, if neither $Y_{\Lambda^{1+\alpha}\setminus \Lambda}=\emptyset$ nor the boundary condition $Y_{\Lambda^{1+\alpha}\setminus \Lambda}$ can be extended to a near triangular lattice configuration in $\Lambda^{1+\alpha}$.

\begin{defin} For $Y\in \cX_*^{\Lambda, z}$, we define the Gibbs distribution in the region $\Lambda \Subset \RRR^2$ with boundary condition $Y$ by the formula

\[\gamma^z_{\Lambda}(F|Y)=\int_{\cX_\Lambda} \1_F(X Y_{\Lambda^c}) e^{-H_{\Lambda,Y}(X)} \Pi^z_\grL(\dd X)/Z^z_{\Lambda, Y},
\]
where $F\in\cA$. Note that $\gamma^z_{\Lambda}(\cdot|Y)$ is a measure on the whole space $(\cX, \cA)$.
\end{defin}
In case of $Y_{\grL^\gra\setminus \grL}\not=\emptyset$, the $\cX_\grL$-marginal of the measure $\gamma^z_{\Lambda}(\cdot|Y)$ is uniform on the configurations in $\cX_\grL$ that extended $Y_{\grL^\gra\setminus \grL}$ to a near triangular lattice configuration in $\Lambda^\alpha$. Otherwise if $Y_{\grL^\gra\setminus \grL}=\emptyset$, then $\gamma^z_{\Lambda}(\cdot|Y)=\delta_{Y_{\grL^c}}$. Note that $(F,Y)\in(\cA, \cX)\mapsto \gamma^z_{\Lambda}(F|Y)$ is a probability kernel from $(\cX, \cA_{\Lambda^c})$ to $(\cX, \cA)$, but the distribution $\gamma^z_{\Lambda}(\cdot|Y)$  has $\delta_{Y_{\grL^c}}$ as its marginal on $(\cX_{\grL^c}, \cA'_{\grL^c})$.

\begin{defin}[infinite-volume Gibbs measure]
A probability measure $P$ on $(\cX,\cA)$ is an an infinite-volume Gibbs measure for $z>0$ if $P(\cX_*^{\Lambda, z})=1$ and
\[
\int f \dd P = \int_{\cX_*^{\Lambda, z}} \frac{1}{Z^z_{\Lambda, Y}}\int_{\cX_\grL} f(XY_{\Lambda^c}) e^{-H_{\Lambda,Y}(X)} \Pi^z_\grL(\dd X) P(\dd Y)
\]
for every $\Lambda\Subset \RRR^2$ and every measurable $f: \cX \to [0,\infty)$. We denote the set of infinite-volume Gibbs measures by $\cG^z$.
\end{defin}

Note that the right hand side in the defining equality is equal to $\EEE_P[\gamma^z_\grL(f| \cdot)]$. Therefore, a measure $P$ is infinite volume Gibbs measure, if and only if $P\gamma^z_\grL=P$ for every $\grL\Subset \RRR^2$, where the product is understood as to take average with $P$ in the second variable of $\gamma^z_\grL$. We can easily see a degenerated measure $\grd_\emptyset\in\cG^z$, however we will be interested in more interesting Gibbs measures. In fact, as soon as $P(\emptyset)=0$ for a measure $P\in\cG^z$, we have that $P$ is supported on hard disk configurations with infinitely many disks.

The Hamiltonian $H$ implements an example of a $k$-nearest neighbor interaction as explained in \cite[Chapter 4.2.1]{DDG12}. Therefore by \cite[Lemma 5.1.]{DDG12}, the kernels $\grg_\grL^z$, $\grg_\grD^z$ for $\grL\subset \grD \Subset \RRR^2$ and $Y\in \cX_*^{\Lambda, z}$ satisfy the consistency conditions $\grg_\grL^z(\cX_*^{\grL, z} | Y)=1$ and $\grg_\grD^z\grg_\grL^z=\grg_\grD^z$, where the product is understood as product of probability kernels.

\subsection{Results}
We show the following generalization\footnote[1]{The wording of Theorem \ref{2Dthm} up to some minor modification in the definition of $H$ was suggested by Franz Merkl in a talk at a conference (Trends in Mathematical Crystallization) held at Warwick University in May 2016} of \cite[Thm. 4.1]{AG14}.

\begin{thm} Let $0<\alpha$ be small enough (such that Lemma \ref{enum} and Theorem \ref{2Dthmenum} hold true for the choice of this $\alpha$). Then for every $2/(\sqrt{3}(1+\gra)^2)<\rho<2/\sqrt{3}$ (the density of centers in the densest packing of disks with diameter 1), there is a measure $P_\rho\in \cap_{z>0} \cG^z$ such that
\begin{enum(i)}
\item \emph{Density = $\rho$}: For any $\grL\Subset \RRR^2$, we have $\EEE_{P_\rho}[N_\grL]=\rho \lambda(\grL)$.
\item \emph{Translational invariance}: The measure $P_\rho$ is translational invariant in any direction in $\RRR^2$, i.e. $P_\rho\circ\theta_r^{-1}=P_\rho$ for any $r\in\RRR^2$.
\item \emph{Long-range orientational order}: Let $x\in X$ be the point with the smallest distance from the origin. It is a.s. unique. We have $P_\rho(N_{A_{1,1+\alpha}}(x)=6)=1$. Choose a random neighbor $y\in  X$ of $x$
(i.e. $1 < |y - x| < 1 + \alpha$) uniformly distributed among all six neighbors. Then as $\rho\uparrow 2/ \sqrt{3}$, the law of $y-x$ w.r.t. $P_\rho$ converges weakly to the uniform distribution on the 6th roots of unity in $\CCC \ \hat= \ \RRR^2$.
\end{enum(i)}
\label{2Dthm}
\end{thm}

Note that by translational invariance of $P_\rho$, property \emph{(iii)} holds when initially picking the closest point $x$ to any reference point $x_0\in\RRR^2$ instead of the origin. Hence the long-range orientational order, as neighbors of $x$ position themselves close to translates of the 6th roots of unity. The choice of $\alpha$ will be made somewhat explicit in the proof of Lemma \ref{enum}. The set of Gibbs measures $\cG^z$ is most likely independent of $z>0$, however we won't pursue the proof of this statement as it leads to geometric considerations that are not in the center of our analysis. 
\subsection{Proofs}
For a configuration $X\in \cX$, we say that $H(X)=0$ if for all $x,y\in X$, we have $|x-y|>1$ and $|X\cap A_{1, 1+\gra}(x)|=6$. This is the same as having $H_{\grL, X}(X)=0$ for any $\grL\Subset\RRR^2$. For a configuration $\emptyset \not=X\in \cX$ with $H(X)=0$, we can define a simplicial complex $K(X)$ consisting of zero, one and two cells defined as follows. The set of zero cells $K_0(X)$ is $X \subset \RRR^2$. The set of one cells $K_1(X)$ are edges between zero cells of distance between $1$ and $1+\gra$, and the two cells are triangles with sides in $K_1(X)$. We will see in the following Lemma, that by definition of $H$ and some geometric considerations, for $\gra$ small enough, the graph defined by the one and two skeleton of this complex is locally, and therefore also globally isomorphic to the triangular lattice $I=\ZZZ+\tau \ZZZ$ with $\tau=e^{\frac{i\pi}{3}}$ with edge set $E=\{\{i,j\}\subset I: |i-j|=1\}$. The set of triangles surrounded by three edges in $E$ is denoted by $\cT$, these are two cells if we regard $I$ as a simplicial complex.

The most important lemma linking the theorem above to \cite[Thm. 4.1]{AG14} is the following.

\begin{lem} With the choice of a small enough $\gra$, we have for any configuration $X \in \cX$ with $H(X)=0$, that the graph defined by the one and two skeletons of $K(X)$ is isomorphic to the triangular lattice $I$. In other words, there is a bijective map $\omega: I\to X$ such that for all $i,j\in I$: $|i-j|=1$ if and only if $|\omega(i)-\omega(j)|\in(1,1+\gra)$.
\label{enum}
\end{lem}

Later on, we will choose $\gra$ small enough such that Lemma \ref{enum} and Theorem \ref{2Dthm} both work for that $\gra$. From the proof of the lemma it will be obvious that the choice of $\gra$ doesn't need to be particularly small for it (and any smaller choice) to work. 

\begin{proof}
We define for $i\in I$ its closest neighborhood $N(i)\subset I$ by $N(i)=\{j\in I: |i-j|\leq 1\}$. Let $X \in \cX$ such that $H(X)=0$. A map $\omega: N(i)\to X$ is called a local isomorphism at $i$ if for all $j,k\in N(i)$, we have $|j-k|=1$ if and only if $|\omega(j)-\omega(k)|\in(1,1+\gra)$. By taking $\gra>0$ small enough, we can ensure that for all $i\in I$ and $x\in X$ there is a local isomorphism $\omega$ at $i$ such that $\omega(i)=x$. To see this, observe that as $\alpha\to 0$, for every $y\in A_{1,1+\gra}(x)$ there are exactly two points $y_1,y_2\in A_{1,1+\gra}(x) \setminus\{y\}$ such that $|y_i-y|\to 1$, for other $z\in A_{1,1+\gra}(x) \setminus\{y\}$, we have $\liminf_{\gra\to 0} |z-y|\geq\sqrt 3$. Since we know that $|X\cap A_{1, 1+\gra}(y)|=6$, a simple geometric consideration related to the kissing problem, gives that $y_1, y_2\in A_{1, 1+\gra}(y)$, since if $y_i \not \in A_{1, 1+\gra}(y)$ for $i\in \{1,2\}$, for $\gra$ small enough there was not enough space to place 6 points in $A_{1, 1+\gra}(y)$ having distance bigger than 1 from each other and from $y_i$. To be more precise, for all $i\in I$ and $x\in X$ there will be twelve such local isomorphisms taking rotations and reflection into account. We fix $\gra$ small enough such that the local isomorphism property holds, since it holds for any small enough $\gra$, we can choose $\gra$ to be smaller than $\sqrt3-1$.

Let us construct a map $\omega: I\to X$ as follows. We fix an arbitrary $x_0\in  X$ and define $\omega|_{N(0)}$ to be one of the six orientational preserving local isomorphism at $0$ with $\omega(0)=x_0$. Fix a spanning tree $T$ of $I$. For each $i\in I$, there is a unique path on nearest neighbors in $T$ connecting $0$ to $i$. Since there are local isomorphism at each pair of points of $I$ and $X$, we can successively, uniquely extend $\omega$ to vertices of $T$ by choosing the unique of the six orientation preserving local isomorphisms that is consistent with $T$. This is to say that if for a neighbor $i$ of $j$ in $T$, we already assigned a point $\omega(i)$ then we already choose a local isomorphism at $i$ with $i\mapsto \omega(i)$. Let us assign $j$ to the point in $X$ which is determined by this local isomorphism. Now, there is only one local isomorphism at $j$, which is consistent with the local isomorphism chosen at $i$ in the sense that $i$ has identical images under the two local isomorphisms. We use this local isomorphism to proceed with the construction and map all neighbors of $j$ in $T$ into $X$.

It remains to show that the map $\omega: I\to X$ is an isomorphism. To conclude $\omega$ is an isomorphism onto its image, we fix a loop $\gamma$ starting and ending in $i\in I$ composed of a path in $T$ and an edge between $i$ and one of its neighbors in $I$ to which it is not connected in $T$. We need to show that the map induced along $\gamma$ with an initial orientational preserving local isomorphism $\omega|_{N(i)}$ at i, maps to a loop in $K(X)$ starting and ending in $\omega(i)$. To this end we can show a seemingly more general but equivalent statement. Take any loop $\gamma=(i_0, i_1, i_2, \dots, i_n)$ at $0\in I$ (i.e. $i_0=i_n=0$) and $x\in X$, fix a local isomorphism at $0$ with $0\mapsto x$ and show that the map induced along $\gamma$ maps $\gamma$ to a loop $\omega(\gamma)$ in $X$ at $x$. Here $\omega$ is a locally defined along the curve $\gamma$.

We can deform the loop $\gamma$ to the boundary of a two cell that contains 0 by successively "removing" two cells that intersect $\gamma$ and are inside of it. By removing a two cell, we mean one of the following. Two subsequent edges $(i_{k-1}, i_{k})$, $(i_{k}, i_{k+1})$ of $\gamma$, we can exchange for the unique edge $(i_{k-1},i_{k+1})$ if $|i_{k-1}-i_{k+1}|=1$, or we can exchange one edge $(i_k, i_{k+1})$ of $\gamma$ for two edges $(i_k, j)$ and $(j,i_{k+1})$ in $I$. For every such transformation of $\gamma$, we obtain a modified $\gamma'$ and a map $\omega'$ that is uniquely determined by the local isomorphism at $i_{k}$ and is the unique extension of the local isomorphism at 0 along $\gamma'$. Note that $\omega=\omega'$ on the domain that they are both defined and $\omega(\gamma)$ is closed if and only if $\omega'(\gamma')$ is. When after removing finitely many two cells, we arrive at $\gamma'=(0, i, j, 0)$ being the boundary of a two cell that contains the origin. Since $\omega'|_{\gamma'}$ should be the unique extension of the local isomorphism at 0 along $\gamma$, we see that $\omega'(\gamma‘)$ is closed and therefore so is $\omega(\gamma)$.

It remains to show that $\omega$ is surjective. Take now a curve $\hat \gamma$ in $K(X)$ from $x_0$ to some $y\in K(X)$. Note that $K(X)$ is a connected graph, as for small enough $\gra$ and $x\not=y$ we can always find a neighbor $z$ of $x$ which is closer to $y$ than $x$. The curve $\gamma$ corresponds to a curve $\gamma$ in $I$ from $0$ to some $i\in I$. Applying the procedure from above to the concatenation of the path from $0$ to $i$ in $T$ and the reverse of $\gamma$, we see that $\omega(i)=y$.
\end{proof}

This lemma can be also proved with the formalism of \v Cech cohomology using the de Rham isomorphism and can be generalized to configurations with point defects (missing points). The usefulness of the \v Cech cohomology and de Rham's theorem was pointed out to us by Franz Merkl. We decided to give another proof using less formalism.

To construct $P_\rho$, we use measures on periodic configurations. For $l>1$ and $n\in\NNN$, let us define measures $P_{n,l}$ on $n$-periodic configurations as in \cite{AG14}. A periodic, enumerated configuration $\omega\in\Omega^{per}_{n, l}$ is a map $I\to\RRR^2$ such that Theorem \ref{2Dthmenum} hold true for this choice of $\alpha$.

\begin{equation}
\omega(i+nj)=\omega(i)+lnj \quad \textrm{for all}\ i, j \in I.
\label{periodic}
\end{equation}

It suffices to define an $n$-periodic, enumerated configuration on a set of $n^2$ representatives $I_n\subset I$ as equation (\ref{periodic}) uniquely defines the configuration on the complement $(I_n)^c$. The event of admissible, $n$-periodic, enumerated configurations $\Omega_{n, l}\subset \Omega^{per}_{n, l}$ is defined by the properties $(\Omega 1)-(\Omega 3)$:

$(\Omega 1) \quad |\omega(i)-\omega(j)|\in (1, 1+\gra)$ for all $\{i,j\}\in E$.

For $\omega\in\Omega$ we define the extension $\hat\omega: \RRR^2\to\RRR^2$ such that $\hat\omega(i)=\omega(i)$ if $i\in I$, and on the closure of any triangle $\triangle\in\cT$, the map $\hat\omega$ is defined to be the unique affine linear extension of the mapping defined on the corners of $\triangle$.

$(\Omega 2)\quad$ The map $\hat\omega: \RRR^2\to\RRR^2$ is injective.

$(\Omega 3) \quad$ The map $\hat\omega$ is orientation preserving, this is to say that $\det(\nabla\hat\omega(x))>0$ for all $\triangle\in\cT$ and $x\in\triangle$ with the Jacobian $\nabla\hat\omega: \cup\cT\to\RRR^{2\times2}$.

Define the set of \emph{admissible, $n$-periodic, enumerated configurations} as

\[
\Omega_{n, l}=\{\omega\in\Omega_{n, l}^{per}\ |\ \omega\ \textrm{satisfies}\ (\Omega1)\textrm{--}(\Omega3)\}.
\]

Let the probability measure $P_{n, l}$ be

\[
P_{n, l}(A)=\frac{\delta_0\otimes\lambda^{I_n \setminus\{0\}} (\Omega_{n, l}\cap A)}{\delta_0\otimes\lambda^{I_n \setminus\{0\}} (\Omega_{n, l})}
\]

for any Borel measurable set $A\in \cF_n=\bigotimes_{i\in I_n}\cB(\RRR^2)$, thus $P_{n, l}$ is the uniform distribution on the set $\Omega_{n, l}$ with respect to the \emph{reference measure} $\delta_0\otimes\lambda^{I_n \setminus\{0\}}$. The first factor in this product refers to the component $\omega(0)$. The parameter $l$ in the definition of $\Omega_{n,l}$ and $P_{n,l}$ controls the density of periodic configurations such that $\rho=\frac{2}{l^2\sqrt{3}}$. We quote Theorem 4.1 from \cite{AG14} which will be the major ingredient of the proof of Theorem \ref{2Dthm}.

\begin{thm} For any $0<\alpha$ small enough one has

\begin{equation}
\lim_{l\downarrow 1}\sup_{n\in\NNN}\sup_{\triangle\in\cT} E_{P_{n,l}}[\ |\nabla\hat\omega(\triangle)-\textnormal{Id} |^2\ ]=0
\end{equation}

with the constant value of the Jacobian $\nabla\hat\omega(\triangle)$ on the set $\triangle\in\cT$.
\label{2Dthmenum}
\end{thm}

We note that the theorem holds for any $\gra\in(0,\sqrt 3 -1)$, however we omit the proof of this which is just a more careful consideration of arguments in the proof of \cite[Theorem 4.1]{AG14} and will refer to small enough $\alpha$. The main observation needed for this explicit range of $\gra$ where the theorem holds is, that the area of triangles with side lengths in the range $[1,\sqrt 3)$ is uniquely minimized by the regular triangle with side length 1. This observation is then utilized like in the similar proof of Theorem \ref{thm1} in the 3D case. We note that Theorem \ref{2Dthmenum} might work with $\gra \geq \sqrt 3 -1$, however looking for the optimal upper bound is not the concern of this paper.  

In the following we construct $P_\rho$ as a limit of translational invariant versions of $P_{n,l}$ and show that this measure is a Gibbs measure in $\cG^z$ for any $z>0$. We follow ideas from \cite{DDG12} to construct a limiting measure. Fix $l>1$ and define the measures $G_n$ on $(\cX, \cA)$ by specifying it's marginal $(G_n)_{\grL_n}$ on $(\cX_{\grL_n},\cA_{\grL_n}')$
\[
(G_n)_{\grL_n}=\left(\frac{1}{\grl( \Lambda_n)}\int_{ \grL_n} \textrm{Im}[P_{n,l}]\circ \theta_r\ \dd r\right)_{\grL_n},
\]
with the image measure $\textrm{Im}[P_{n,l}]$ of $P_{n,l}$ under the map $\textrm{Im}:  \omega\mapsto \{\omega(x) :  x\in I\}$ and the domain $\Lambda_n=l\{x+y\tau : x,y\in[-n/2, n/2)\}$. The averaging over $r\in \grL_n$ is necessary to obtain a translational invariant measure on the torus, since $\omega(0)=0$ holds $P_{n,l}-$a.s.. The measure $G_n$ is then defined by having i.i.d. projections on the sets $\{ \grL_n+inl\}_{i\in I}$, which form a tiling of $\RRR^2$. In order to have translational invariant probability measures on $(\cX, \cA)$, we consider the averaged measures
\[
\hat G_{n}=\frac{1}{\grl( \Lambda_n)} \int_{ \Lambda_n}G_{n}\circ \theta_r\ \dd r
\]
By definition and the periodicity of $G_n$, $\hat G_n$ are translational invariant. We will show that the sequence $(\hat G_n)_{n\in\NNN}$ is tight in the \emph{topology of local convergence} on translational invariant probability measures on $\cX$ generated by $P\to \int f \dd P$ for functions $f$ that are $\cA_\grL$-measurable for some $\Lambda\Subset \RRR^2$. Such functions we call local and denote the set of local functions by $\cL$.

The only difference to the definitions after Lemma 5.1. in \cite{DDG12} are in the nature of the measures $(G_n)_{\grL_n}$. In our case $(G_n)_{\grL_n}$ are measures that inherit geometric constraints from the structure of $P_{n,l}$ that are defined on toruses of different size. In \cite{DDG12} on the contrary, the authors use a measures $G^{z}_{\grL_n, \bar\omega}$ that have fixed boundary condition $\bar\omega$ on the complement of $\grL_n$.

For a shift invariant probability measure $P$ on $(\cX, \cA)$ and $\grL\Subset \RRR^2$ define the measure $P_{\grL}:=P\circ  \textrm{pr}_{\grL}^{-1}$ and the \emph{relative entropy} w.r.t. $\Pi_{\grL}^z$ as

\[
I(P_{\grL} | \Pi_{\grL}^z):=
\begin{cases}
\int f \ln f \dd \Pi_{\grL}^z \quad &\textrm{if } P_{\grL}<<\Pi_{\grL}^z \textrm{ with density } f  \\
\infty \quad &\textrm{otherwise}
\end{cases}.
\]
The \emph{specific entropy} of $P$ w.r.t. $\Pi^z$ is then defined by
\[
I(P):= \lim_{n\to \infty} \frac{1}{\grl(\grD_n)} I(P_{\grD_n} | \Pi_{\grD_n}^z),
\]
where $\Delta_n\Subset\RRR^2$ is a cofinite increasing sequence of sets. We refer to \cite{G88} and \cite{GZ93} for existence and properties of the specific entropy. We will set $z=1$ and compute entropies relative to $\Pi_{\grD_n}^1$. By \cite[Proposition 2.6]{GZ93}, the sublevel sets of $I$ are sequentially compact in the topology of local convergence. Therefore, we only need to show that the specific entropies of the measures $\{\hat G_n\}_{n\in\NNN}$ are bounded by some constant. We start with a proposition that provides lower bound on the partition sum.

\begin{prop}
For all $\gra\in (0,1]$ and $l\in(1,1+\gra)$, there is an $r=r(\gra, l)\in(0,1/2)$ such that for $n\in\NNN$, we have
\begin{eqnarray}
\delta_0\otimes\lambda^{I_n\setminus\{0\}}(\Omega_{n,l})\geq (\pi r^2)^{|I_n|-1}.
\end{eqnarray}
\label{Zbound}
\end{prop}

\begin{proof}
For $r>0$, we define, like in (3.2) in \cite{HMR}, the set of configurations which are close to the scaled, enumerated, standard configuration $\omega_l(i)=li$ for $i\in I$:

\begin{eqnarray}
S_{n,l,r}=\{\omega\in \Omega_{n,l}^{per}\ | \  |\omega(i)-\omega_l(i)|<r \textrm{ for all } i\in I \}.
\end{eqnarray}

For sufficiently small $r>0$, depending on $\gra$ and $l$, we conclude, like in the proof of \cite[Lemma 3.1]{HMR}, that $S_{n,l,r}\subset\Omega_{n,l}$. To prove this inclusion, we have to show the properties $({\Omega 1})\textrm{--}({\Omega 3})$ for all $\omega\in S_{n,l,r}$. Let us compute for $(i,j)\in E$ and $\omega\in S_{n,l,r}$:

\begin{eqnarray*}
||\omega(i)-\omega(j)|-l|&=&||\omega(i)-\omega(j)|- |\omega_l(i)-\omega_l(j)| |\\
&\leq& |\omega(i)-\omega_l(i)| + |\omega(j)-\omega_l(j)|<2r.
\end{eqnarray*}

If we choose $2r<\max\{l-1, 1+\gra-l\}<1$, then $\omega$ satisfies $({\Omega 1})$. Condition $({\Omega 2})$ is a consequence of the inequality $\langle v, \nabla\hat\omega(x)v\rangle>0$ for all $v\in\RRR\setminus\{0\}$, and for all $x\in\RRR^2$ where $\hat\omega$ is differentiable. This inequality holds for small enough r since $\nabla\hat\omega$ is close to the identity uniformly on $\RRR^2$. Hence $\hat\omega$ is a bijection onto its image. Here we applied a theorem from analysis which states that a $\cC^1$-map $f$ from an open convex domain $U\subset\RRR^n$ into $\RRR^n$ with $\langle v, \nabla f(x)v\rangle>0$ for all $v\in\RRR^n\setminus\{0\}$ and $x\in U$ is a diffeomorphism onto its  image. However, $\nabla\hat\omega(x)$ is only piecewise differentiable, but on the straight line $L$ connecting $x,y\in \RRR^2$ with $x\not=y$, there are only finitely many points $z\in \RRR^2\cap L$ where the curve $(\hat\omega(ty+(1-t)x))_{t\in (0,1)}$ is not differentiable. Assume that $\langle v, \nabla\hat\omega(x)v\rangle>0$ holds whenever $\hat\omega$ is differentiable in $x$. The curve is piecewise linear, and on each of these pieces, the derivative of the curve forms an acute angle with $y-x$, therefore the curve cannot be closed. Thus, the condition $({\Omega 2})$ is satisfied in the case of a sufficiently small $r$. Furthermore, condition $({\Omega 3})$ is satisfied by $\omega_l$, therefore also by $\omega$ if $r$ is sufficiently small. Hence $S_{n,l,r}\subset\Omega_{n,l}$ for some $r\in(0,1/2)$, and we conclude
 
\begin{eqnarray*}
\delta_0\otimes\lambda^{I_n\setminus\{0\}}(\Omega_{n,l})\geq\delta_0\otimes\lambda^{I_n\setminus\{0\}}(S_{n,l,r})=(\pi r^2)^{|I_n|-1}
\end{eqnarray*}

where the last equality is obtained by integrating over each $\omega(i)$ with $i\not= 0$ successively along a fixed spanning tree of $I_n$ which gives a factor $\pi r^2$, and considering that $\omega_l(0)=0$ and that the measure $\delta_0\otimes\lambda^{I_n\setminus\{0\}}$ fixes $\omega(0)=0$.
\end{proof}

\begin{prop} The set $\{I(\hat G_n): n\in\NNN \}$ is bounded, thus the set $\{\hat G_n: n\in\NNN \}$ is sequentially compact in the topology of local convergence. Therefore, there is a sequence $n_k\to\infty$ and a shift invariant measure $P_\rho$ on $(\cX, \cA)$ such that $\lim_{k\to\infty}\int f \dd G_{n_k}=\int f \dd P_\rho$ for any $f\in \cL$.
\label{locallimit}
\end{prop}
\begin{proof} As also noted in the proof of \cite[Proposition 5.3]{DDG12}, the definition of $\hat G_n$ implies that
\[
I^z(\hat G_n)=\frac{1}{\grl( \Lambda_n)}I\left((G_n)_{ \grL_n}|\Pi_{\grL_n}^1\right).
\]
The relative entropy $I\left((G_n)_{ \grL_n}|\Pi_{\grL_n}^1\right)$ can be explicitly computed as follows. The measure $(G_n)_{ \grL_n}$ is supported on configurations that have $n^2$ points in $ \grL_n$ and if $\grL_n$ is folded into a torus, then each point $x$ has exactly six neighbors in the annulus $A_{1,1+\alpha}(x)$ around it and no points closer than distance one. These configurations $\cX_{n,l}$ are images of enumerated configurations $\cX_{n,l}=(\textrm{Im}\ \Omega_{n,l})_{ \grL_n}$. By Lemma \ref{enum}, $(G_n)_{\grL_n}$ is the uniform distribution on these configurations with respect to $\Pi_{\grL_n}^1$. The density of $(G_n)_{\grL_n}$ w.r.t. $\Pi_{\grL_n}^1$ is given by $f=\1_{\cX_{n,l}}/ \Pi_{\grL_n}^1(\cX_{n,l})$. To find the constant $\Pi_{\grL_n}^1(\cX_{n,l})$ more explicitly, consider the expectation
\[\Pi^1_{\grL_n}[g]=e^{-\grl(\grL_n)} \sum_{k=0}^\infty \int_{\grL_n^k}\frac{1}{k!}g(\{x_1, \dots, x_k\}) \ \grl^k|_{\grL_n^k}(\dd x_1, \dots, \dd x_k)
\]
Consequently, we have  
\[\Pi_{\grL_n}^1(\cX_{n,l})=\frac{e^{-\grl(\grL_n)}}{n^2}\grl(\grL_n) \ \delta_0\otimes\lambda^{I_n \setminus\{0\}} (\Omega_{n, l}).
\]
This follows since a factor $\frac{e^{-\grl(\grL_n)}}{(n^2)!}$ comes from the density of $ \Pi_{\grL_n}^1$ conditioned on $n^2$ points with respect to $\ \grl^{(n^2)}|_{\grL_n^{(n^2)}}(\dd x_1, \dots, \dd x_n^2)$. Then conditioned on the position of $x_1$, the volume of the allowed configurations by their shift invariance on the torus is $(n^2-1)!\ \delta_0\otimes\lambda^{I_n \setminus\{0\}} (\Omega_{n, l})$, furthermore the first point can be distributed uniformly in $\grL_n$. The relative entropy is $I\left((G_n)_{ \grL_n}|\Pi_{\grL_n}^1\right)=-\ln\left(\Pi_{\grL_n}^1(\cX_{n,l})\right)$ and the specific entropy can be bounded using Proposition \ref{Zbound} and $\grl( \Lambda_n)=n^2 l^2 \sqrt{3}/2$ for big enough $n$, we obtain
\begin{align*}
I\left((G_n)_{ \grL_n}\right)=-\frac{\ln\left(\Pi_{\grL_n}^1(\cX_{n,l})\right)}{\grl( \Lambda_n)}&= 1+\frac{n^2}{\grl( \Lambda_n)}-\frac{\ln \left( \grl( \Lambda_n)\right)}{\grl( \Lambda_n)}-\frac{\ln\left( \delta_0\otimes\lambda^{I_n \setminus\{0\}} (\Omega_{n, l})\right)}{\grl( \Lambda_n)}\\
&\leq1+\frac{n^2}{\grl( \Lambda_n)}-\frac{\ln \left( \grl( \Lambda_n)\right)}{\grl( \Lambda_n)}-\frac{|I_n-1| \ln(\pi r^2)}{\grl( \Lambda_n)}\\
&\leq1+\frac{2-2\ln (\pi r^2)}{l^2 \sqrt{3}}.
\end{align*}
\end{proof}

The next proposition shows that $P_\rho$ is an infinite-volume Gibbs measure. Note that $\hat G_n$ and $\Lambda_n$ depend on $l>1$ which we fixed previously.

\begin{prop}
The measure $P_\rho$ is an infinite-volume Gibbs measure $P_\rho\in\cap_{z>0} \cG^z$.
\label{Gibbs}
\end{prop}
\begin{proof}
Fix $\Lambda \Subset \RRR^2$, $z>0$ and $\rho<2/\sqrt{3}$ large enough such that $2/(\sqrt{3}(1+\gra)^2)<\rho$ where $\gra$ is such that Lemma \ref{enum} holds with that $\alpha$. Let $l>1$ such that $\rho=2/(l^2\sqrt{3})$. For $X\in \cX$, let $\widetilde X_n$ be the periodic extension of $X_{\grL_n}$ to $\cX$, i.e. $\widetilde X_n=\cup_{i\in I} X_{\grL_n}+lni$. Let $\kappa>0$ be so big such that $\grL^{\kappa}\setminus \grL$ contains a connected ring of triangles from $K_2(\widetilde X_n)$ for $G_n$-almost all $X$ for all $n\in\NNN$. Consequently, for all $n\in\NNN$ large enough such that $\grL^{\kappa}\subset\grL_n$, the number of points in $\grL$ conditioned on $X_{\grL^c}$ is $G_n$-almost surely determined by the configuration in $\grL^{\kappa}\setminus \grL$. The measure $(G_n)_{\Lambda_n}$ is the uniform distribution of enumerable, allowed configurations with $n^2$ points on the torus. By Lemma \ref{enum}, the conditional distribution of $X_\grL$ given $X_{\grL^c}$ under $G_n$ is therefore the uniform distribution on configurations $X_\grL$ such that $H_{\grL, X_{\grL^c}}(X_\grL)=0$. Uniform distribution makes sense, as the number of points in $\grL$ is almost surely constant with respect to the conditioned measure. Therefore, the factorized version of the conditional distribution of $G_n$ given $\cA_{\grL^c}$ is given by $\gamma_\grL(\cdot | \cdot)$, this is to say that
\begin{equation}
G_n(F)=\int_{\cX}\gamma_{\grL}(F|Y) G_n(\dd Y)
\label{consist}
\end{equation}
for any $F\in \cA$ and $n\in\NNN$ big enough for $\grL^{\kappa}\subset\grL_n$. Since $z$ is fixed, we can omitted it as a superscript in $\gamma^z$.
  
The rest of the proof is as the proof of \cite[Prop. 5.5.]{DDG12}. Define $\grL^\circ_n:=\{r\in\RRR^2 :\grL^{\kappa}+r\subset \grL_n \}$ and the (subprobability) measures
\[
\bar G_n:= \frac{1}{|\grL_n|}\int_{\grL^\circ_n} G_n\circ \theta_r^{-1} \ \dd r.
\]
Then $\int f \dd \hat G_n - \int f \dd \bar G_n\to 0$ by the same argument as in \cite[Lemma 5.7]{GZ93}, therefore $P_\rho$ can also be seen as an accumulation point of the sequence $(\bar G_n)$. Let $F \in \cup_{\grD\Subset \RRR^2}\cA_{\grD}$ be a local set, using (\ref{consist}), we obtain for $r\in\grL^\circ_n$
\[
G_n\circ \theta_r^{-1}(F)=\int_{\cX} \gamma_{\grL}(F|Y)G_n\circ \theta_r^{-1}(\dd Y).
\]
Therefore averaging over $r\in\grL^\circ_n$ gives
\begin{equation}
\bar G_n(F)=\int_{\cX} \gamma_{\grL}(F|Y)\bar G_n(\dd Y).
\label{Gibbseq}
\end{equation}
Since the integrand on the right is a local function of $Y$, we can set $n=n_k$ and let $k\to \infty$, that gives (\ref{Gibbseq}) for $P_\rho$ instead of $\bar G_n$. Since local sets generate the $\grs$-algebra $\cA$, (\ref{Gibbseq}) holds for $P_\rho$ and $F\in\cA$, which by monotone convergence shows that $P_\rho$ is an infinite-volume Gibbs measure.
\end{proof}
\begin{proof}[Proof of Theorem \ref{2Dthm}]
In Propositions \ref{Gibbs} and \ref{locallimit}, we showed the existence of a translational invariant measure $P_\rho\in\cap_{z>0}\cG^z$ which is the local limit of the measures $(G_{n_k})_{k\geq 1}$, therefore $P_\rho$ satisfies property (ii). Property (i) holds as it can be expresses by a local function and $\EEE_{G_{n_k}}[|X\cap B|]=\rho \grl(B)$ for any $k\geq 1$ by the periodic boundary conditions. Similarly, property (iii) can be expressed by local functions depending on $\{x_0, x_1, ..., x_6\}\cap \grL_n$, where $x_0$ is the closest random point to the origin and $x_i$ is $i$'th closest point to $x_0$. For $n$ large enough we have $G_{n_k}(|\{x_0, x_1, ..., x_6\}\cap \grL_n|=7)=1$ for any $k\geq 1$ and therefore $P_\grr(|\{x_0, x_1, ..., x_6\}\cap \grL_n|=7)=1$. By Theorem \ref{2Dthmenum} we have
\begin{equation}
\lim_{\grr\uparrow 2/\sqrt{3}}\sup_{k\geq 1}\EEE_{G_{n_k}}\left[ \sum_{i=1}^6|\nabla\hat\omega(\triangle_i)-\textnormal{Id} |^2\right]=0,
\label{2Dthmenumresult}
\end{equation}
where $\{\triangle_i\}_{1\leq i\leq 6}$ are the random six triangles in $\cT$ such that one of their vertices is mapped to $x_0$ under $\omega$. Let $f:\CCC^6\to \RRR$ be continuous, bounded and permutation invariant. We use the natural identification of topological spaces $\CCC\hat = \RRR^2$. Let $y_i=x_i-x_0$. By continuity of $f$, there is a constant $c>0$ such that
\begin{equation}
\left|f(y_1, \dots, y_6)-f(e^{i\pi/3}, e^{i2\pi/3}, \dots, e^{i2\pi}) \right|\leq c \sum_{i=1}^6|\nabla\hat\omega(\triangle_i)-\textnormal{Id} |^2
\label{fdiffbound}
\end{equation}
$G_{n_k}$-a.s. for any $k\geq 1$. Combining equations (\ref{2Dthmenumresult}) and (\ref{fdiffbound}), we obtain that
\[
\lim_{\grr\uparrow 2/\sqrt{3}} \EEE_{P_\grr}\left[\left|f(y_1, \dots, y_6)-f(e^{i\pi/3}, e^{i2\pi/3}, \dots, e^{i2\pi}) \right|\right]=0
\]
which concludes the proof of property (iii).
\end{proof}

\textbf{Acknowledgment}: I'd like to thank Franz Merkl for fruitful discussions and his suggestions.

\end{document}